\documentclass[a4paper,12pt,reqno]{amsart}

\usepackage{amsfonts,amssymb,amscd,amsmath,latexsym,amsbsy,enumerate,a4wide,verbatim}

\newtheorem{thm}{Theorem}[section]
\newtheorem{cor}[thm]{Corollary}
\newtheorem{lemma}[thm]{Lemma}
\newtheorem{prop}[thm]{Proposition}

\theoremstyle{definition}
\newtheorem{remark}[thm]{Remark}

\numberwithin{equation}{section}

\def\al{\alpha}
\def\be{\beta}
\def\ga{\gamma}
\def\de{\delta}

\def\la{\lambda}

\def\si{\sigma}
\def\vp{\varphi}

\def\Om{\Omega}

\def\La{\Lambda}

\def\R{\mathbb{R}}
\def\C{\mathbb{C}}
\def\N{\mathbb{N}}

\def\hf{\frac12}

\def\cH{\mathcal H}

\def\cU{\mathcal U}

\def\cV{\mathcal V}

\newcommand{\rphis}[5]{\,_{#1}\vp_{#2} \left( \genfrac{.}{.}{0pt}{}{#3}{#4}
\ ;#5 \right)}

\title{Spectral decomposition and matrix-valued orthogonal polynomials}
\author{Wolter Groenevelt, Mourad E.H. Ismail, Erik Koelink}
\date{\today}
\address{Technische Universiteit Delft, DIAM, EWI, 
Postbus 5031,
2600 GA Delft,
the Netherlands}
\email{w.g.m.groenevelt@tudelft.nl}
\address{
Department of Mathematics, University of Central Florida, Orlando FL 32816, USA,
and King Saud University, Riyadh, Saudi
Arabia.}
\email{ismail@math.ucf.edu}
\address{Radboud Universiteit Nijmegen, IMAPP, FNWI, Heyendaalseweg 135, 6525 AJ Nijmegen,
 the Netherlands}
\email{e.koelink@math.ru.nl}
 
\begin{document}
\begin{abstract} The relation between the spectral decomposition of a self-adjoint operator which is 
realizable as a higher order recurrence operator and matrix-valued orthogonal polynomials is investigated. 
A general construction of such operators from scalar-valued orthogonal polynomials is presented. Two 
examples of matrix-valued orthogonal polynomials  with explicit orthogonality relations and three-term 
recurrence relation are presented, which both can be considered as $2\times 2$-matrix-valued analogues of 
subfamilies of Askey-Wilson polynomials.
\end{abstract}

\maketitle


\section{Introduction}\label{sec:intro}

Matrix-valued orthogonal polynomials date back to the 1950's in the work of M.G.~Krein, see e.g. 
references in \cite{DamaPS}, \cite{DuraLR}. More recently, matrix-valued orthogonal polynomials are 
studied from an analytic point of view. In particular, analogues of many classical results in the theory 
of ordinary (scalar-valued) orthogonal polynomials have been generalized to the situation of 
the matrix-valued orthogonal polynomials, such as e.g. the three-term recurrence relation, the 
spectral theorem (Favard), theorems of Markov, Blumenthal, etc., see the overviews \cite{DamaPS},
\cite{DuraLR} and references given there. 
Many examples of the general theory of matrix-valued orthogonal polynomials are motivated by 
matrix-valued differential equations, see also \cite{GrunT}. Some of these examples are motivated 
from the well-known families of orthogonal polynomials in the Askey scheme \cite{KoekS}, so the 
matrix-valued weight function is given by the scalar weight function times a suitable matrix-valued 
function. So in this case matrix-valued analogues of classical orthogonal polynomials, such as 
Jacobi, Laguerre and Hermite polynomials, are obtained. 
This theory sofar gives matrix-valued analogues of hypergeometric orthogonal polynomials. Very 
little is known about matrix-valued analogues of $q$-orthogonal polynomials.  

Another way of obtaining matrix-valued orthogonal polynomials is from group theory using 
matrix-valued spherical functions. An important case study has been given by Gr\"unbaum, Pacharoni 
and Tirao \cite{GrunPT}, in which they obtain matrix-valued orthogonal polynomials from the 
symmetric pair 
$(SU(3),U(2))$ by studying eigenfunctions to invariant matrix-valued differential operators.
Again these matrix-valued orthogonal polynomials are analogues of a subfamily of Jacobi polynomials. In 
\cite{KoelRvP}, \cite{KoelRvP2} a different approach to such a group-theoretic approach has led to 
matrix-valued Chebyshev polynomials including relevant group theoretic interpretations 
of the construction, the three-term recurrence relation, weight function, differential 
equations, etc, using the symmetric pair 
$(SU(2)\times SU(2), SU(2))$. Again, in these cases the weight function resembles the 
corresponding scalar weight function times a suitable matrix-valued function. Again, 
no $q$-matrix-valued orthogonal polynomials have yet emerged from this approach. 

In this paper we discuss a new way to obtain matrix-valued orthogonal polynomials with an explicit 
three-term recurrence relation as well as explicit orthogonality relations. 
In the examples it is clear that the weight function is not of the form of a classical weight 
function times a matrix-valued function. The idea is to look for the spectral decomposition 
of a self-adjoint operator which can also be realized as a higher order recurrence operator. 
In order to motivate the construction, we first note that if we consider an operator which can be realized 
as a $2N+1$-recurrence operator, the case $N=0$ corresponds to eigenfunctions. 
The case $N=1$ is the case of 
the $J$-matrix (or tridiagonalization) method, which is used in physics to determine the 
spectrum of certain physically 
relevant operators, see \cite{IsmaK-AAM}, \cite{IsmaK-toappear} and references given there. 
In \cite{IsmaK} a more general 
method to obtain suitable tridiagonalizable operators is discussed. In this paper we restrict 
ourselved to 
self-adjoint operators that can be realized as $5$-term recurrence operators and for which we 
have an explicit spectral decomposition. We show in Theorem \ref{thm:genorth2t2MVOP} how this 
gives rise to $2\times 2$-matrix-valued orthogonal polynomials with an explicit (matrix-valued) 
three-term recurrence relation and explicit matrix-valued orthogonality relations. 
Because of computability reasons we stick to the $2\times 2$-case, but we expect that it is possible to 
extend to larger size matrices. In Section \ref{sec:exrhots} we discuss an explicit example 
with an easy matrix-valued three-term recurrence relation, but an involved, but explicit, 
expression for the matrix-valued weight function. In Section \ref{sec:genexamples} we 
discuss a general set-up, which is motivated by \cite{IsmaK}, and we work out a specific 
example in Section \ref{ssec:exlittleqJacobipols} which is related to the example in 
\cite[\S 4]{IsmaK}. This motivates us to view the family of matrix-valued orthogonal 
polynomials discussed in the 
example of Section \ref{ssec:exlittleqJacobipols} as analogues of a subfamily of Askey-Wilson polynomials. 

As is well-known, it is very hard in general to obtain explicit expressions for the 
orthogonality measures or weights for orthogonal polynomials defined by a three-term recurrence 
relation. The cases of associated classical orthogonal polynomials 
(in the Askey-scheme \cite{KoekS}) amply demonstrates this point, see e.g. \cite{IsmaR} for 
the case of two families of the associated Askey-Wilson polynomials. It is therefore remarkable 
that we can obtain in this 
setting an explicit, even though complicated, expression for both the weight function 
and the three-term recurrence relations for the 
$2\times 2$-matrix-valued orthogonal polynomials in the examples considered in this paper. 
Moreover, to our best knowledge this is the first instance of matrix-valued orthogonal 
polynomials that can be considered as matrix-valued orthogonal polynomials in a yet-unknown 
(possible) $q$-scheme of matrix-valued orthogonal polynomials, see \cite{KoekS} for the scalar 
case. Note that we do not have explicit expressions for the $2\times 2$-matrix-valued 
orthogonal polynomials, and it would be of interest to obtain such expressions for these 
polynomials in terms of (yet to be developed) matrix-valued basic hypergeometric series of 
higher type, see Tirao \cite{Tira} for the matrix-valued analogue of the classical 
hypergeometric function.


\section{Matrix-valued orthogonal polynomials from 5-term operators}\label{sec:5trrMVOP}

In this section we study the relation between a self-adjoint operator realizable as $5$-term 
operator and corresponding $2\times 2$-matrix-valued orthogonal polynomials. The three-term 
matrix-valued recurrence relations for these polynomials follow from this realization of the 
operator, whereas the orthogonality relations for these polynomials follow from the spectral 
decomposition of the operator. The precise relation is given in Theorem \ref{thm:genorth2t2MVOP}. 

We assume that we have an operator $T$ on a Hilbert space $\cH$ of functions. For $T$ we 
typically consider a second-order difference or differential operator. We assume that $T$ 
has the following properties;
\begin{enumerate}[(a)]
\item $T$ is (a possibly unbounded) self-adjoint operator on $\cH$ (with domain $D$ in 
case $T$ is unbounded);
\item there exists an orthonormal basis $\{f_n\}_{n=0}^\infty$ of $\cH$ so that 
$f_n\in D$ in case $T$ is unbounded and so that there exist sequences 
$(a_n)_{n=0}^\infty$, $(b_n)_{n=0}^\infty$, $(c_n)_{n=0}^\infty$ of numbers with 
$a_n>0$, $c_n\in\R$,  for all $n\in \N$ so that 
\begin{equation}\label{eq:Tis5term}
T\, f_n\, = \, a_n f_{n+2} \, + \, b_n f_{n+1} \, + \, c_n f_n \, + \, 
\overline{b_{n-1}}f_{n-1} \, + \, a_{n-2}f_{n-2}.
\end{equation}
\end{enumerate}
In (b) we follow the convention that $a_{-1}=a_{-2}=b_{-1}=0$.

Next we assume that we have a suitable spectral decomposition of $T$. 
We assume that the spectrum is simple or at most of multiplicity $2$, and we leave it to 
the reader to extend to higher order spectra. We assume that the double spectrum is contained in 
$\Om_2\subset \si(T) \subset \R$, and the simple spectrum is contained in 
$\Om_1=\si(T)\setminus \Om_2\subset \R$. Consider functions $f$ defined on $\si(T)\subset \R$ so that 
$f\vert_{\Om_1} \colon \Om_1\to \C$ and $f\vert_{\Om_2} \colon \Om_2\to \C^2$.
We let $\si$ be a Borel measure on $\Om_1$ and $V\, \rho$ a $2\times 2$-matrix-valued measure 
on $\Om_2$ as in \cite[\S 1.2]{DamaPS}, so $V\colon \Om_2 \to \text{M}_2(\C)$ maps into the 
positive semi-definite matrices and $\rho$ is a positive Borel measure on $\Om_2$.
Next we consider the weighted Hilbert space $L^2(\cV)$ of such functions for which
\[
\int_{\Om_1} |f(\la)|^2\, d\si(\la) \, + \, \int_{\Om_2} f^\ast(\la) 
V(\la) f(\la)\, d\rho(\la) \, < \, \infty
\]
and we obtain $L^2(\cV)$ by modding out by the functions of norm zero. The inner product is given by
\[
\langle f, g\rangle \, = \, \int_{\Om_1} f(\la)\overline{g(\la)}\, d\si(\la) \, + \, 
\int_{\Om_2} g^\ast(\la) V(\la) f(\la)\, d\rho(\la).
\]

The final assumption is then
\begin{enumerate}
 \item[(c)] there exists a unitary map $U\colon \cH \to L^2(\cV)$ so that $UT=MU$, where $M$ is 
 the multiplication operator on $L^2(\cV)$. 
\end{enumerate}

Under the assumptions (a), (b), (c) we link the spectral measure to an orthogonality measure 
for matrix-valued orthogonal polynomials.
Apply $U$ to the $5$-term expression \eqref{eq:Tis5term} for $T$ on the basis 
$\{f_n\}_{n=0}^\infty$, so that 
\begin{multline}\label{eq:5trrforUfn}
\la (Uf_n)(\la)\, = \, a_n (Uf_{n+2})(\la) \, + \, b_n (Uf_{n+1})(\la) \, 
\\ +\, c_n (Uf_n)(\la) \, + \, \overline{b_{n-1}}(Uf_{n-1})(\la) \, + \, a_{n-2}(Uf_{n-2})(\la)
\end{multline}
to be interpreted as an identity in $L^2(\cV)$. Restricted to $\Om_1$ \eqref{eq:5trrforUfn} 
is a scalar identity, and restricted to $\Om_2$ the components of 
$Uf(\la) = (U_1f(\la),U_2f(\la))^t$ satisfy \eqref{eq:5trrforUfn}. 

In general, a $2N+1$-term recurrence relation can be solved using $N\times N$-matrix-valued 
orthogonal polynomials, see Dur\'an and Van Assche \cite{DuraVA}. Working out the details 
for $N=2$, we see that we have to generate the 
$2\times 2$-matrix-valued polynomials by 
\begin{equation}\label{eq:3trr2by2MVOP}
\begin{split}
 \la\, P_n(\la) \, &= \, A_n\, P_{n+1}(\la)
\, +\, B_n P_n(\la)
\, +\, A_{n-1}^\ast P_{n-1}(\la), \\ 
A_n\, &= \, \begin{pmatrix} a_{2n} & 0 \\ b_{2n+1} & a_{2n+1} \end{pmatrix}, \qquad
B_n\, = \, \begin{pmatrix} c_{2n} & b_{2n} \\ \overline{b_{2n}} & c_{2n+1} \end{pmatrix}
\end{split}
\end{equation}
with initial conditions $P_{-1}(\la)=0$ and $P_0(\la)$ is a constant non-singular matrix, 
which we take to be the identity, so $P_0(\la)=I$. Note that $A_n$ is a non-singular 
matrix and $B_n$ is a Hermitian matrix for all $n\in\N$. Then the $\C^2$-valued functions 
\begin{equation*}
\cU_n(\la)\, =\, \begin{pmatrix} Uf_{2n}(\la) \\ Uf_{2n+1}(\la) \end{pmatrix}, \qquad
\cU^1_n(\la)\, =\, \begin{pmatrix} U_1f_{2n}(\la) \\ U_1f_{2n+1}(\la) \end{pmatrix}, \qquad
\cU^2_n(\la)\, =\, \begin{pmatrix} U_2f_{2n}(\la) \\ U_2f_{2n+1}(\la) \end{pmatrix}
\end{equation*}
satisfy \eqref{eq:3trr2by2MVOP} for vectors for $\la\in \Om_1$ in the first case and 
for $\la\in \Om_2$ in the last cases. Hence, 
\begin{equation}\label{eq:Un=PnU0}
\cU_n(\la)\, = \, P_n(\la)  \cU_0(\la),  \qquad
\cU_n^1(\la)\, = \, P_n(\la)  \cU^1_0(\la), \quad
\cU_n^2(\la)\, = \, P_n(\la)  \cU^2_0(\la), 
\end{equation}
where the first holds $\si$-a.e. and the last two hold $\rho$-a.e.
We can now state the orthogonality relations for the matrix-valued orthogonal polynomials. 

\begin{thm}\label{thm:genorth2t2MVOP} With the assumptions (a), (b), (c) as given above, 
the $2\times 2$-matrix-valued polynomials $P_n$ generated by \eqref{eq:3trr2by2MVOP} 
and $P_{-1}(\la)=0$, $P_0(\la)=I$  satisfy 
\[
\int_{\Om_1} P_n(\la) \, W_1(\la)\, P_m(\la)^\ast\, d\si(\la)\, + \, 
\int_{\Om_2} P_n(\la) \, W_2(\la)\, P_m(\la)^\ast\, d\rho(\la)
=\, \de_{nm} I 
\]
where 
\begin{equation*}
W_1(\la)\, = \,  \begin{pmatrix} |Uf_0(\la)|^2 & Uf_0(\la)\overline{Uf_1(\la)} \\
\overline{Uf_0(\la)}Uf_1(\la) & |Uf_1(\la)|^2 \end{pmatrix}, \qquad
\si-\text{a.e.} 
\end{equation*}
and 
\begin{equation*}
W_2(\la)\, = \,  \begin{pmatrix} \langle Uf_0(\la), Uf_0(\la)\rangle_{V(\la)} 
& \langle Uf_0(\la), Uf_1(\la)\rangle_{V(\la)} \\
\langle Uf_1(\la), Uf_0(\la)\rangle_{V(\la)} & \langle Uf_1(\la), Uf_1(\la)\rangle_{V(\la)}
\end{pmatrix}, \qquad
\rho-\text{a.e.} 
\end{equation*}
where $\langle x, y\rangle_{V(\la)} = x^\ast V(\la) y$. 
\end{thm}

Theorem \ref{thm:genorth2t2MVOP} can be phrased more compactly, and then the generalization 
to self-adjoint operators $T$ realizable as higher order recurrence relations can be phrased 
compactly as well. Since we stick to the situation with the assumptions (a), (b), (c), the 
multiplicity of $T$ cannot be higher than $2$. Note that the matrices $W_1(\la)$ and 
$W_2(\la)$ are Gram matrices.
In particular,  
$\det(W_1(\la)) = 0$ for all $\la$. So the weight matrix $W_1(\la)$ 
is semi-definite positive with eigenvalues 
$0$ and $\text{tr}(W(\la))= |Uf_0(\la)|^2 + |Uf_1(\la)|^2>0$. Note that
\[
\text{ker}(W(\la)) \, = \, \C \begin{pmatrix} 
\overline{Uf_1(\la)} \\ -\overline{Uf_0(\la)} \end{pmatrix} \, = \, 
\begin{pmatrix} Uf_{0}(\la) \\ Uf_{1}(\la)\end{pmatrix}^\perp, \quad
\text{ker}(W(\la)-\text{tr}(W(\la))) = \C \begin{pmatrix} Uf_{0}(\la) \\ Uf_{1}(\la)\end{pmatrix}
\]
Moreover, $\det(W_2(\la)) = 0$ if and only if $Uf_0(\la)$ and $Uf_1(\la)$ are multiples of each other. 

\begin{proof} Start using the unitarity 
\begin{equation}\label{eq:startoforthoPn}
\begin{split}
\, \de_{nm} \begin{pmatrix} 1 &  0 \\ 0 & 1\end{pmatrix} \, &= \,
 \begin{pmatrix} \langle f_{2n},f_{2m}\rangle_{\cH} &  \langle f_{2n},f_{2m+1}\rangle_{\cH} 
\\ \langle f_{2n+1},f_{2m}\rangle_{\cH} & \langle f_{2n+1},f_{2m+1}\rangle_{\cH}\end{pmatrix}\\\, &= \, 
\begin{pmatrix} \langle Uf_{2n},Uf_{2m}\rangle_{L^2(\cV)} &  \langle Uf_{2n},Uf_{2m+1}\rangle_{L^2(\cV)} 
\\ \langle Uf_{2n+1},Uf_{2m}\rangle_{L^2(\cV)} & 
\langle Uf_{2n+1},Uf_{2m+1}\rangle_{L^2(\cV)}\end{pmatrix}\\
\end{split}
\end{equation}
Split each of the inner products on the right hand side of \eqref{eq:startoforthoPn} 
as a sum over two integrals, one over $\Om_1$ and the other over $\Om_2$. First the integral 
over $\Om_1$ equals
\begin{equation}\label{eq:startoforthoPnOm1}
\begin{split}
&\, \begin{pmatrix} \int_{\Om_1}Uf_{2n}(\la) \overline{Uf_{2m}(\la)}\, d\si(\la) &  
\int_{\Om_1}Uf_{2n}(\la) \overline{Uf_{2m+1}(\la)}\, 
d\si(\la) \\ \int_{\Om_1}Uf_{2n+1}(\la) \overline{Uf_{2m}(\la)}\, d\si(\la) &
\int_{\Om_1}Uf_{2n+1}(\la) \overline{Uf_{2m+1}(\la)}\, d\si(\la)\end{pmatrix} \\
=&\, \int_{\Om_1}\begin{pmatrix} Uf_{2n}(\la) \overline{Uf_{2m}(\la)} &  
Uf_{2n}(\la) \overline{Uf_{2m+1}(\la)} \\ 
Uf_{2n+1}(\la) \overline{Uf_{2m}(\la)} &
Uf_{2n+1}(\la) \overline{Uf_{2m+1}(\la)}\end{pmatrix} \, d\si(\la)\\
=&\, \int_{\Om_1} \begin{pmatrix} Uf_{2n}(\la) \\ Uf_{2n+1}(\la)\end{pmatrix}
\begin{pmatrix} Uf_{2m}(\la) \\ Uf_{2m+1}(\la)\end{pmatrix}^\ast \, d\si(\la)\\
=&\, \int_{\Om_1} P_n(\la) \begin{pmatrix} Uf_{0}(\la) \\ Uf_{1}(\la)\end{pmatrix}
\begin{pmatrix} Uf_{0}(\la) \\ Uf_{1}(\la)\end{pmatrix}^\ast P_m(\la)^\ast \, d\si(\la) \\
=&\, \int_{\Om_1} P_n(\la) W_1(\la) P_m(\la)^\ast \, d\si(\la),
\end{split}
\end{equation}
where we have used \eqref{eq:Un=PnU0}.  
For the integral over $\Om_2$ we write $Uf(\la) = (U_1f(\la), U_2f(\la))^t$ and 
$V(\la) = (v_{ij}(\la))_{i,j=1}^2$, so that the integral over $\Om_2$ can be written as 
\begin{equation}\label{eq:startoforthoPnOm2a}
\begin{split}
&\, \sum_{i,j=1}^2 \int_{\Om_2} \begin{pmatrix} U_jf_{2n}(\la) v_{ij}(\la) \overline{U_if_{2m}(\la)} 
& U_jf_{2n}(\la) v_{ij}(\la) \overline{U_if_{2m+1}(\la)} \\  
U_jf_{2n+1}(\la) v_{ij}(\la) \overline{U_if_{2m}(\la)} & 
U_jf_{2n+1}(\la) v_{ij}(\la) \overline{U_if_{2m+1}(\la)}\end{pmatrix} d\rho(\la) \\
=&\, \sum_{i,j=1}^2 \int_{\Om_2} \begin{pmatrix} U_j f_{2n}(\la) \\ U_j f_{2n+1}(\la)\end{pmatrix}
\begin{pmatrix} U_i f_{2m}(\la) \\ U_i f_{2m+1}(\la)\end{pmatrix}^\ast v_{ij}(\la)\, d\rho(\la) \\
=&\, \sum_{i,j=1}^2 \int_{\Om_2} P_n(\la) \begin{pmatrix} U_j f_{0}(\la) \\ U_j f_{1}(\la)\end{pmatrix}
\begin{pmatrix} U_i f_{0}(\la) \\ U_i f_{1}(\la)\end{pmatrix}^\ast  P_m(\la)^\ast v_{ij}(\la)\, d\rho(\la) 
\,\\ =&\,  \int_{\Om_2} P_n(\la) W_2(\la)  P_m(\la)^\ast \, d\rho(\la), \\
\end{split}
\end{equation}
where we have used \eqref{eq:Un=PnU0} again and with 
\begin{equation}\label{eq:startoforthoPnOm2b}
\begin{split}
W_2(\la) \, &=\, \sum_{i,j=1}^2 \begin{pmatrix} U_j f_{0}(\la) \\ U_j f_{1}(\la)\end{pmatrix}
\begin{pmatrix} U_i f_{0}(\la) \\ U_i f_{1}(\la)\end{pmatrix}^\ast v_{ij}(\la) \,\\ 
&=\, \sum_{i,j=1}^2 v_{ij}(\la) \begin{pmatrix} U_j f_{0}(\la)\overline{U_if_0(\la)} 
& U_j f_{0}(\la)\overline{U_if_1(\la)} \\
U_j f_{1}(\la)\overline{U_if_0(\la)} & 
U_j f_{1}(\la)\overline{U_if_1(\la)}
\end{pmatrix} \\
\, &=\, \begin{pmatrix} (Uf_0(\la))^\ast V(\la) Uf_0(\la) & (Uf_1(\la))^\ast V(\la) Uf_0(\la) \\
(Uf_0(\la))^\ast V(\la) Uf_1(\la) & (Uf_1(\la))^\ast V(\la) Uf_1(\la) \end{pmatrix}
\end{split}
\end{equation}
and putting \eqref{eq:startoforthoPnOm1} and \eqref{eq:startoforthoPnOm2a}, 
\eqref{eq:startoforthoPnOm2b} into \eqref{eq:startoforthoPn} proves the result.
\end{proof}

In case we additionally assume $T$ is bounded, so that the measures $\si$ and $\rho$ 
have compact support, the coefficients in \eqref{eq:Tis5term} and \eqref{eq:3trr2by2MVOP} 
are bounded. In this case the corresponding moment problem is determinate, see 
\cite[Thm.~2.11]{DamaPS}, and Theorem \ref{thm:genorth2t2MVOP} gives the explicit 
expression for the weight function. 

\begin{remark}
Assume that $\Om_1=\si(T)$ or $\Om_2=\emptyset$, so that $T$ has simple spectrum. Then 
\begin{equation}\label{eq:defL2Wsigma}
\mathcal{L}^2(W_1d\si) \, = \, \{ f\colon \R \to \C^2 \mid 
\int_\R f(\la)^\ast W_1(\la) f(\la) \, d\si(\la)<\infty\} 
\end{equation}
has the subspace of null-vectors 
\begin{multline*}
\mathcal{N}\, = \, \{ f\in \mathcal{L}^2(W_1d\si) \mid 
\int_\R f(\la)^\ast W(\la) f(\la) \, d\si(\la) = 0\} \\ 
\, = \, \{ f\in \mathcal{L}^2(W_1d\si) \mid f(\la) = c(\la) \begin{pmatrix} 
\overline{Uf_1(\la)} \\ -\overline{Uf_0(\la)} \end{pmatrix} \text{ $\si$-a.e.}\}, 
\end{multline*}
where $c$ is a scalar-valued function. In this case  
$L^2(\cV)=\mathcal{L}^2(W_1d\si)/\mathcal{N}$. Note that 
$\mathcal{U}_n\colon \R\to L^2(W_1d\si)$ is completely determined by $Uf_0(\la)$,
 which is a restatement of $T$ having simple spectrum. 
From Theorem \ref{thm:genorth2t2MVOP} we see that 
\[
\langle P_n(\cdot)v_1, P_m(\cdot)v_2\rangle_{L^2(W_1d\si)} \, = \, 
\de_{nm} \langle v_1, v_2\rangle
\]
so that $\{ P_n(\cdot)e_i\}_{i\in\{1,2\}, n\in \N}$ is linearly independent in 
$L^2(W_1d\si)$ for any basis $\{e_1,e_2\}$ of
$\C^2$. 
\end{remark}


\section{A general class of examples}\label{sec:genexamples}

In \cite{IsmaK} we have studied a general procedure to obtain self-adjoint 
tridiagonalizable operators, and in this section  we show how to extend this to obtain 
self-adjoint operators  which can be realized as $5$-term recurrence. This brings us back to 
the situation of Section \ref{sec:5trrMVOP}, hence leading to $2\times 2$-matrix-valued  
orthogonal polynomials. Of course, we still need to obtain the 
spectral decomposition of such operators as well. 
We extend \cite[\S 2]{IsmaK} in Section \ref{ssec:SA-5termdiag-operators} and we present an 
example of the construction using little $q$-Jacobi polynomials in Section 
\ref{ssec:exlittleqJacobipols}. The analogue of the Jacobi polynomials is rather involved, 
in particular the spectral decomposition, and this is worked out in \cite{GroeK}.

\subsection{Self-adjoint penta-diagonalisable operators}\label{ssec:SA-5termdiag-operators}

Let $\mu$ and $\nu$ be positive Borel measures with finite moments on the real line $\R$
so that $\mu$ is absolutely continuous with respect to $\nu$. 
Let $r = \frac{d\mu}{d\nu}$ be the Radon-Nikodym derivative, so $r\geq 0$.
We assume that we have a (possibly unbounded) self-adjoint operator 
$L$ on $L^2(\mu)$ preserving the space of polynomials in $L^2(\mu)$ and the 
existence of an orthonormal basis $\{\Phi_n\}_{n\in\N}$ of $L^2(\mu)$
of polynomial eigenfunctions of $L$, so 
$L\Phi_n = \la_n \Phi_n$, $\la_n\in\R$. 
Moreover, we assume the existence of an orthonormal basis $\{\phi_n\}_{n=0}^\infty$ of polynomials in 
of $L^2(\nu)$ such that for all $n\in\N$ 
\begin{equation}\label{eq:phicombinationofPhi}
\phi_n \, = \, \al_n\, \Phi_n \, + \, \be_n \, \Phi_{n-1} \, + \, \ga_n \, \Phi_{n-2}, 
\qquad \al_n, \be_n, \ga_n\in \R
\end{equation}
(with the convention $\be_0=\ga_0=\ga_1=0$). We assume that the polynomials 
are dense in $L^2(\mu)$ and $L^2(\nu)$. 
Finally we assume that the Radon-Nikodym derivative $r$ is a polynomial, 
necessarily at most of degree $2$ by \eqref{eq:phicombinationofPhi}. 
We denote by $M(r)$ and $M(x)$  the multiplication operator by $r$ and by $x$.   

\begin{lemma}\label{lem:Tsiis5trr} 
$T^\rho = M(r) \bigl(L + \rho\bigr)$, $\rho\in\R$,
is a symmetric five-diagonal operator on $L^2(\nu)$ with respect to 
the orthonormal basis $\{\phi_n\}_{n=0}^\infty$;
\[
T^\rho\, \phi_n \, = \,  \, a_n \phi_{n+2} \, + \, \tilde{b}_n \phi_{n+1} 
\, + \, \tilde{c}_n \phi_n \, + \, \tilde{b}_{n-1}\phi_{n-1} \, + \, a_{n-2}\phi_{n-2}
\]
where 
\begin{equation*}
\begin{split}
a_n\, &= \, \al_n\ga_{n+2}(\la_n+\rho), \quad
\tilde{b}_n\, = \, \al_n\be_{n+1}(\la_n+\rho)+\be_n(\la_{n-1}+\rho)\ga_{n+1}, \\ 
\tilde{c}_n\, &= \, \al_n^2(\la_n+\rho)+\be_n^2(\la_{n-1}+\rho) + \ga_n^2(\la_{n-2}+\rho).
\end{split}
\end{equation*}
\end{lemma}

\begin{proof} This is completely analogous to \cite[\S 2.1]{IsmaK}.
\end{proof}

Since the orthonormal basis $\{\phi_n\}_{n=0}^\infty$ of $L^2(\nu)$ consists of polynomials, we have 
\begin{equation}\label{eq:3trforgeneralcase}
x\phi_n(x)\, =\, \theta_n \phi_{n+1}(x)\, + \, \xi_n\phi_n(x) \, + \, \theta_{n-1} \phi_{n-1}(x), 
\end{equation}
for $\theta_n,\xi_n\in\R$,  $\theta_n\not=0$ for all $n\in\N$ and the convention $\theta_{-1}=0$. 

\begin{cor}\label{cor:lemTsiis5trr} 
$T^{\rho,\tau} = M(r) \bigl(L + \rho\bigr)+\tau M(x)$, $\rho,\tau\in\R$,
is a symmetric five-diagonal operator on $L^2(\nu)$ with respect to the 
orthonormal basis $\{\phi_n\}_{n=0}^\infty$;
\[
T^{\rho,\tau}\, \phi_n \, = \,  \, a_n \phi_{n+2} \, + \, b_n \phi_{n+1} \, 
+ \, c_n \phi_n \, + \, b_{n-1}\phi_{n-1} \, + \, a_{n-2}\phi_{n-2}
\]
where $b_n =  \tilde{b}_n  + \tau\theta_n$, 
$c_n =  \tilde{c}_n +  \tau\xi_n$, and the notation as in Lemma \ref{lem:Tsiis5trr}.
\end{cor}

Note that in case $L$ is a second-order differential or difference operator, then so is $T$.  
However, the  coefficients of $T$ get more complicated and in order to carry through the programme 
of Section \ref{sec:5trrMVOP} we need to be able to calculate the  spectral decomposition of 
$T^{\rho,\tau}$ for suitable $\rho$, $\tau$ as well in another way. 

\begin{remark}\label{rmk:extension} It is clear that we can extend this to higher order 
recurrences. So if we assume $r$ to be a polynomial of degree $N$ and the 
recursion \eqref{eq:phicombinationofPhi} to have $N+1$ terms, we end with a $2N+1$-recursion 
for the operator in Lemma \ref{lem:Tsiis5trr} and Corollary \ref{cor:lemTsiis5trr}. 
\end{remark}


\subsection{Example: case of little $q$-Jacobi polynomials}\label{ssec:exlittleqJacobipols}

We work out the details of the general programme of Section \ref{ssec:SA-5termdiag-operators} for 
the case of the little $q$-Jacobi polynomials, cf. \cite[\S 4]{IsmaK}. 
Let, as usual, $0<q<1$, and we follow standard notation for basic hypergeometric series as in
\cite{GaspR}, see also \cite{Isma}, \cite{KoekS}.

The little $q$-Jacobi polynomials are
\begin{equation}\label{eq:deflittleqJacobipols}
p_n(x)\, = \, p_n(x;a,b;q) \, = \, \rphis{2}{1}{q^{-n},\, abq^{n+1}}{aq}{q, qx}
\end{equation}
with leading coefficient
\begin{equation}
l_n(a,b)\, = \, (-1)^n q^{-\hf n(n-1)} \frac{(abq^{n+1};q)_n}{(aq;q)_n}
\end{equation}
and for $0<a<q^{-1}$, $b<q^{-1}$ the little $q$-Jacobi polynomials 
satisfy the orthogonality relations 
\begin{equation}\label{eq:orthorellittleqJacobipols}
\begin{split}
&\qquad\qquad  \sum_{k=0}^\infty p_n(q^k)p_m(q^k)\, w_k(a,b) \, =\, \de_{nm} h_n(a,b), \\
&w_k(a,b)\, =\, (aq)^k \frac{(bq;q)_k\, (aq;q)_\infty}{(q;q)_k\, (abq^2;q)_\infty}, \quad
h_n(a,b)\, = \, \frac{1-abq}{1-abq^{n+2}} \frac{(q,bq;q)_n}{(aq,abq;q)_n} (aq)^n
\end{split} 
\end{equation}
normalising $h_0(a,b)=1$. The little $q$-Jacobi polynomials satisfy 
\begin{equation}\label{eq:differenceoplittleqJacobipols}
\begin{split}
&L^{(a,b)} p_n(\cdot;a,b;q) \, = \, \la_n \, p_n(\cdot;a,b;q), \qquad 
\la_n = \la_n(a,b) = q^{-n}(1-q^n)(1-abq^{n+1})\\
&\bigl(L^{(a,b)}f\bigr)(x)\, = \, 
\frac{a(bqx-1)}{x}\bigl(f(qx)-f(x)\bigr) + \frac{x-1}{x}\bigl(f(\frac{x}{q})-f(x)\bigr)
\end{split} 
\end{equation}

In the context of Section \ref{ssec:SA-5termdiag-operators} we take $L^2(\mu)$, 
respectively $L^2(\nu)$, to be the weighted $L^2$-space corresponding to the case $(aq,bq)$, 
respectively $(a,b)$. Note that
\begin{equation}\label{eq:relwabandwaqbq}
w_k(aq,bq) = r(q^k) w_k(a,b), \quad r(x) = K^{-1} x(1-bqx), \ \ 
K = \frac{(1-aq)(1-bq)}{(1-abq^2)(1-abq^3)}>0. 
\end{equation}

In the context of Section \ref{ssec:SA-5termdiag-operators} we see that we can 
give a five-term recursion formula for the operator $T^{\rho,\tau}$ defined by 
\begin{equation}\label{eq:TsitauforlittleqJacobipols}
\begin{split}
&\bigl( T^{\rho,\tau}f\bigr)(x) \, = \, \frac{aq}{K}(1-bqx)(bq^2x-1)\bigl( f(qx)-f(x)\bigr) 
\\ &\, +\, \frac{1}{K}(1-bqx)(x-1)\bigl( f(x/q)-f(x)\bigr)\, +\, x(\frac{\rho}{K}(1-bqx)+\tau)f(x).
\end{split}
\end{equation}
In order to apply the link to $2\times 2$-matrix-valued orthogonal polynomials we 
need to give the spectral decomposition of $T^{\rho,\tau}$ on $L^2(\nu)$ in another way. 

\begin{prop} \label{prop:TspectralcontdualqHahnpls}
Assume $q^{-1} > b > 0$.
The operator $T^{\rho,\tau}$ for $\rho =(1+q\sqrt{ab})(1+q^2\sqrt{ab})$, 
$\tau = \frac{1}{K}\bigl( q\sqrt{ab}(3+2q+bq^2)-bq(1+aq)\bigr)$ is a bounded self-adjoint 
operator with explicit spectral decomposition given by $U\colon L^2(\nu) \to L^2(\si)$ and 
$UT=MU$, where $M$ is multiplication by $2x/(\sqrt{aq}+(1-K)/\sqrt{aq})$ and $\si$ is the 
normalized orthogonality measure for the continuous dual $q$-Hahn polynomials 
\cite{KoekS} with 
parameters $(A,B,C)=(\sqrt{qb},\sqrt{qb}, q\sqrt{qb})$, and $U$ is given by 
\[
U\colon \frac{\de_{q^k}}{\sqrt{w_k(a,b)}}\, \mapsto \, P_n(\cdot;  \sqrt{qb},\sqrt{qb}, q\sqrt{qb}\mid q)
\]
using the orthonormal polynomials on the right hand side. 
\end{prop}

\begin{remark}  Recall that in \cite{IsmaK} we can introduce an additional degree of freedom, 
which is not possible in Proposition \ref{prop:TspectralcontdualqHahnpls}. On the other hand, 
considering more generally second-order difference operators on $L^2(\nu)$ we can introduce an 
additional degree of freedom in the parameters of the continuous dual Hahn polynomials, but then 
we have no longer a nice explicit expression for the $5$-term recurrence as in Section 
\ref{sec:5trrMVOP}.
\end{remark}

\begin{proof} Let $V\colon \ell^2(\N)\to L^2(\nu)$, $e_k\mapsto \de_{q^k}/\sqrt{w_k(a,b)}$, 
be the unitary operator identifying  the Hilbert space $\ell^2(\N)$ with standard orthonormal 
basis $\{e_k\}_{k\in\N}$ with the weighted $L^2(\nu)$ space for the little $q$-Jacobi polynomials.  
Let $J^{\rho,\tau} = V^\ast T^{\rho,\tau}V$, then 
\begin{equation*}
\begin{split}
&\Bigl(\sqrt{aq}+\frac{1}{\sqrt{aq}}+ \frac{-K}{\sqrt{aq}}\Bigr)J^{\rho,\tau}e_k\, =\, 
\tilde\al_k e_{k+1} \, + \, \tilde\be_k e_k \, + \, \tilde\al_{k-1} e_{k-1}  \\
& \tilde\al_k \, = \, (1-bq^{k+1})\sqrt{(1-q^{k+1})(1-bq^{k+1})}, \\
& \tilde\be_k \, = \, q^k\Bigl( bq\sqrt{aq}(1+q) +\frac{1}{\sqrt{aq}}(1+bq)-
\frac{\rho}{\sqrt{aq}}+
\frac{K\tau}{\sqrt{aq}} \Bigr) 
+ q^{2k}\Bigl( \frac{bq}{\sqrt{aq}}(\rho-1)-b^2q^3\sqrt{aq}\Bigr)
\end{split}
\end{equation*}
Comparing this Jacobi operator to the three-term recurrence relation for the orthonormal 
continuous dual $q$-Hahn polynomials with parameters $(A,B,C)$ as in \cite[\S 3.3]{KoekS}, 
we see that we need $\{ AB, AC, BC\} = \{ bq, bq^2, bq^2\}$ to get the right expression 
for  $\tilde\al_k$. Since $b\not= 0$ we get $A=B$, $C=qB$, and because of symmetry we 
obtain $(A,B,C) = (\sqrt{bq},\sqrt{bq},q\sqrt{bq})$. 

In order to match the value of $\be_k$ to the orthonormal continuous $q$-Hahn polynomials  
with these parameters we require 
\begin{gather*}
 bq\sqrt{aq}(1+q) +\frac{1}{\sqrt{aq}}(1+bq)-\frac{\rho}{\sqrt{aq}}+
\frac{K\tau}{\sqrt{aq}} \, =\, A + B + C+ ABC, \\
\frac{bq}{\sqrt{aq}}(\rho-1)-b^2q^3\sqrt{aq} \, = \, ABC(1+q^{-1})
\end{gather*}
which determines the choice for $\rho$ and $\tau$. Then 
$\Bigl(\sqrt{aq}+\frac{1}{\sqrt{aq}}+ \frac{-K}{\sqrt{aq}}\Bigr)J^{\rho,\tau}$ 
has continuous spectrum $[-2,2]$, which gives the statement on $U$.
\end{proof}

Now that we have determined for which values of $(\rho,\tau)$ we have an explicit 
spectral decomposition in 
Proposition \ref{prop:TspectralcontdualqHahnpls}, we have to work out the coefficients 
in Corollary \ref{cor:lemTsiis5trr} in this case. 
We start with \eqref{eq:phicombinationofPhi} in this case, or 
equivalently
\begin{equation}\label{eq:expansionlittleaJacobi}
p_n(x;a,b;q) = a_{n,n} p_n(x;aq,bq;q) + a_{n,n-1}  p_{n-1}(x;aq,bq;q) + 
a_{n,n-2} p_{n-2}(x;aq,bq;q).
\end{equation}
By comparing leading coefficients in \eqref{eq:expansionlittleaJacobi} 
we obtain 
\begin{equation}
a_{n,n} \, = \,  \frac{(1-aq^{n+1})(1-abq^{n+1})(1-abq^{n+2})}{(1-aq)(1-abq^{2n+1})(1-abq^{2n+2})}.
\end{equation}
Using the orthogonality and \eqref{eq:relwabandwaqbq} we obtain
\begin{equation*}
\begin{split}
a_{n-2,n} h_{n-2}(aq,bq) \, &=\,  \sum_{k=0}^\infty p_n(q^k;a,b;q) p_{n-2}(q^k;aq,bq;q) r(q^k) w_k(a,b) \\
\, &=\, \text{lc}(r) \frac{l_{n-2}(aq,bq)}{l_n(a,b)}\, h_n(a,b)
\end{split}
\end{equation*}
using the expansion of $p_{n-2}(\cdot;aq,bq;q) r(\cdot)$ in terms of little 
$q$-Jacobi polynomials with parameters $(a,b)$. This gives 
\begin{equation}\label{eq:expressionann-2}
a_{n,n-2} \, = \,  \frac{-bq^{n+2}}{K}
 \frac{(1-q^{n-1})(1-q^{n})(1-bq)(1-bq^n)}{(1-abq^2)(1-abq^3)(1-abq^{2n-1})(1-abq^{2n})}.
\end{equation}
Note that \eqref{eq:expressionann-2} is not clear from the general connection 
coefficient formula for little $q$-Jacobi polynomials due to Andrews and Askey, see  
\cite[Ex.~1.33]{GaspR}. The coefficient $a_{n,n-1}$ can be obtained by comparing 
coefficients of $x^{n-1}$ on both sides. This gives, after a straightforward calculation,
\begin{equation}\label{eq:expressionannmin1}
a_{n,n-1}\, =\, q^{1-n} \frac{(1-q^n)(1-abq^{n+1})}{(1-q)(1-aq)}
\left( \frac{1-aq^n}{1-abq^{2n}}-\frac{1-aq^{n+1}}{1-abq^{2n+2}} \right)
\end{equation}
Using the orthonormal version we find that in this example the coefficients in 
\eqref{eq:phicombinationofPhi} are 
\begin{equation}\label{eq:defalnbenganlittleqJacobi}
\begin{split}
\al_n  &= 
\frac{q^{\hf n} (1-abq^{n+2})\sqrt{(1-abq^2)(1-abq^3)(1-abq^{n+1})
(1-aq^{n+1})(1-bq^{n+1})}}{(1-abq^{2n+1})(1-abq^{2n+2})\sqrt{(1-abq^{n+4})(1-aq)(1-bq)}}\\
\be_n  &= q^{-\hf n} a^{-\hf} \frac{\sqrt{(1-abq^2)(1-abq^3)(1-q^n)(1-abq^{n+1})
(1-abq^{n+2})}}{(1-q)\sqrt{(1-abq^{n+3})(1-aq)(1-bq)}} \\
&\qquad\times
\left( \frac{1-aq^n}{1-abq^{2n}}-\frac{1-aq^{n+1}}{1-abq^{2n+2}} \right)\\
\ga_n  &= \frac{-bq^{\frac32 n}}{aK} \frac{\sqrt{(1-q^{n-1})(1-q^n)
(1-aq)(1-aq^n)(1-bq)(1-bq^n)}}{(1-abq^{2n-1})(1-abq^{2n})}
\\
\end{split}
\end{equation}

Finally, we need the three-term recurrence relation for the orthonormal 
little $q$-Jacobi polynomials, which corresponds to 
\eqref{eq:3trforgeneralcase} with explicit values 
\begin{equation}\label{eq:3trrlittleqJacobipols}
\begin{split}
\theta_n \, &= \,q^n \frac{\sqrt{aq(1-aq^{n+1})(1-bq^{n+1})(1-q^{n+1})(1-abq^{n+1})(1-abq^{n+2})}}
{(1-abq^{2n+1})(1-abq^{2n+2})\sqrt{1-abq^{n+3}}} \\
\xi_n \, &=\,  \frac{q^n\, (1-aq^{n+1})(1-abq^{n+1})}{(1-abq^{2n+1})(1-abq^{2n+2})} + 
\frac{aq^n \, (1-q^{n})(1-bq^{n})}{(1-abq^{2n})(1-abq^{2n+1})}.
\end{split}
\end{equation}

We next want to use Theorem \ref{thm:genorth2t2MVOP} with the spectral 
decomposition $U$ given by Proposition 
\ref{prop:TspectralcontdualqHahnpls}, so that we assume the situation of 
Proposition \ref{prop:TspectralcontdualqHahnpls}. 
The spectrum is simple, so that $\Om_2=\emptyset$. It remains to 
calculate $U\phi_0$ and $U\phi_1$. Keeping track of normalization we have 
\begin{equation*}
\begin{split}
(U\phi_n)(\cos t) \, =& \, \frac{\sqrt{a(aq,bq,q;q)_\infty}}{\sqrt{(abq^2;q)_\infty}} (bq^2;q)_\infty
\\ &\qquad \times \sum_{k=0}^\infty \frac{q^{\hf k}\, p_n(q^k;a,b;q)}{(q, bq^2;q)_k} 
p_k(\cos t; \sqrt{qb}, \sqrt{qb}, q\sqrt{qb} \mid q)
\end{split}
\end{equation*}
where we have used the standard notation, see \cite{KoekS}, for the continuous dual $q$-Hahn polynomials. 
Using one of the standard generating functions, see \cite[(3.3.15)]{KoekS},  
and $p_1(q^k;a,b;q) = 1 -q^k \frac{(1-abq)}{(1-aq)}$ we find
\begin{equation}
\begin{split}
&F_0(\cos t) \, =\, (U\phi_0)(\cos t) \, = \, \\
&\qquad \qquad \frac{\sqrt{a(aq,bq,q;q)_\infty}}{\sqrt{(abq^2;q)_\infty}} 
\frac{(bq^2;q)_\infty (q\sqrt{b};q)_\infty}{(e^{it}\sqrt{q};q)_\infty}\, 
\rphis{2}{1}{\sqrt{qb}e^{it}, q\sqrt{qb}e^{it}}{bq^2}{q, \sqrt{q}e^{-it}} \\
&F_1(\cos t) \, =\, (U\phi_1)(\cos t) \, = \, 
\frac{\sqrt{a(aq,bq,q;q)_\infty}}{\sqrt{(abq^2;q)_\infty}} (bq^2;q)_\infty \\
&\qquad \qquad\times \Bigl( \frac{ (q\sqrt{b};q)_\infty}{(e^{it}\sqrt{q};q)_\infty}\, 
\rphis{2}{1}{\sqrt{qb}e^{it}, q\sqrt{qb}e^{it}}{bq^2}{q, \sqrt{q}e^{-it}} 
- \\ & \qquad \qquad\qquad \qquad \frac{(1-abq)}{(1-aq)} 
\frac{ (q^2\sqrt{b};q)_\infty}{(e^{it}q\sqrt{q};q)_\infty}\, 
\rphis{2}{1}{\sqrt{qb}e^{it}, q\sqrt{qb}e^{it}}{bq^2}{q, q^{\frac32}e^{-it}} \Bigr) 
\end{split}
\end{equation}
We summarize this situation in the following Proposition \ref{prop:MVOPlittleqJacobipols}, 
using the explicit expression for the orthogonality measure $d\si$ of the continuous dual 
$q$-Hahn polynomials, see \cite[\S 3.3]{KoekS}.

\begin{prop}\label{prop:MVOPlittleqJacobipols}
Define the coefficients $a_n$, $b_n$ and $c_n$ as in Corollary \ref{cor:lemTsiis5trr} 
with the explicit values for $\al_n$, $\be_n$, $\ga_n$ as in \eqref{eq:defalnbenganlittleqJacobi}, 
$\la_n$ as \eqref{eq:differenceoplittleqJacobipols}, $\theta_n$, $\xi_n$ as in 
 \eqref{eq:3trrlittleqJacobipols} and 
$\rho$ and $\tau$ as in Proposition \ref{prop:TspectralcontdualqHahnpls}. 
The $2\times 2$-matrix-valued orthogonal polynomials generated by the three-term 
recurrence relation \eqref{eq:3trr2by2MVOP} with initial conditions
$P_{-1}(\la)=0$, $P_0(\la)=I$ satisfy the orthogonality relations
\[
\frac{(q,qb, q^2b,q^2b;q)_\infty}{2\pi} \int_0^\pi P_n(\cos t) W_1(\cos t) P_m(\cos t)^\ast\, 
\frac{(e^{\pm 2it};q)_\infty }{(qbe^{\pm it}, q^2be^{\pm it}, q^2be^{\pm it};q)_\infty}
dt \, = \, \de_{nm}I
\]
with 
\begin{equation*}
W_1(\cos t) = 
\begin{pmatrix} |F_0(\cos t)|^2 & F_0(\cos t) F_1(\cos t) \\
F_0(\cos t) F_1(\cos t) & |F_1(\cos t)|^2
\end{pmatrix}.
\end{equation*}
\end{prop}

In view of \cite[\S 4]{IsmaK} we view the $2\times 2$-matrix-valued polynomials of 
Proposition \ref{prop:MVOPlittleqJacobipols} as matrix-valued analogue of 
(a subfamily) Askey-Wilson polynomials. 
The case $b\leq 0$ can be dealt with similarly, where the case $b=0$ allows for 
additional degrees of freedom.


\section{Example: spectral decomposition of an operator arising from quantum groups}\label{sec:exrhots}

In an influential paper \cite{Koor} Koornwinder has introduced a special element 
$\rho_{\tau, \si}$ in the quantum $SU(2)$ group. In this context it is important 
the have the action of this element in an infinite dimensional representation as an 
explicit $5$-term recurrence relation. On the other hand, the spectral decomposition
of the corresponding operator has been solved in \cite{KoelV} exploiting the special 
case $\si\to\infty$ as an intermediate step. So the spectral decomposition of the $5$-term 
recurrence is completely known, and by the set-up of Theorem \ref{thm:genorth2t2MVOP} we 
obtain orthogonality relations for $2\times 2$-matrix-valued orthogonal polynomials with 
explicit coefficients for the three-term recurrence relation. The resulting 
Proposition \ref{prop:rhostMVOP} describes the weight function  explicitly in terms of 
${}_2\vp_1$-series. 

Throughout this section we assume $\si,\tau\in\R$.
In this case the Hilbert space is $\cH=\ell^2(\N)$ with standard orthonormal 
basis $\{ f_n\}_{n=0}^\infty$. The operator $T$ corresponds to the operator 
$\pi_\phi(\rho_{\tau,\si})$ of \cite[\S 6]{KoelV}; explicitly in the notation 
of \eqref{eq:Tis5term} we have in this case 
\begin{equation}\label{eq:rhostanbncc}
\begin{split}
a_n\, &= \, \frac12 \sqrt{(1-q^{2n+2})(1-q^{2n+4})}, \\
b_n\, &= \, \frac12 iq^{n+1}\sqrt{1-q^{2n+2}}\bigl( e^{i\phi}(q^{-\si}-q^{\si})
+ e^{-i\phi}(q^{-\tau}-q^{\tau})\bigr) \\
c_n\, &= \, q^{1+2n}\bigl( \cos(2\phi) - \frac12 (q^{-\si}-q^{\si})(q^{-\tau}-q^{\tau})\bigr)
\end{split}
\end{equation}
Note that $a_{-1}=a_{-2}=b_{-1}=0$. Moreover, we have a symmetry  
$(\si,\tau,\phi)\leftrightarrow (\tau,\si,-\phi)$ and $(\si,\tau,\phi)\leftrightarrow 
(-\si,-\tau,\phi+\pi)$. So we can  assume $\si\geq \tau$ and $\si\geq -\tau$. From 
\cite[\S 6]{KoelV} we deduce that $T$ has absolutely continuous spectrum $[-1,1]$ 
of multiplicity $2$ and discrete spectrum (possibly empty) of multiplicity $1$ at 
$\Sigma_-\cup \Sigma_+$, where, using the notation $\mu(x) =\frac12(x+x^{-1})$,  
\begin{equation}\label{eq:rhostdiscretespectrum}
\begin{split}
\Sigma_- \, =&\, \{ \mu(-q^{1-\si-\tau+2k}) \, \mid\, k\in \N, \ q^{1-\si-\tau+2k}>1 \}  \\
\Sigma_+ \, =&\, \{ \mu(q^{1-\si+\tau+2k}) \, \mid\, k\in \N, \ q^{1-\si+\tau+2k}>1 \}.
\end{split}
\end{equation}
From \cite[\S 6]{KoelV} we can read off $L^2(\cV)$. Assume that $\si+\tau\leq 1$, 
$\si-\tau\leq 1$, so that there is no discrete spectrum. Then $V$ is a diagonal 
matrix with the orthonormal measure for 
the Al-Salam--Chihara polynomials with parameters $(q^{1+\si-\tau}, -q^{1-\si-\tau})$, 
respectively 
$(q^{1-\si+\tau}, -q^{1+\si+\tau})$, on the $(1,1)$-entry, respectively the $(2,2)$-entry. Explicitly,
$f\colon [-1,1]\to \C^2$ is in $L^2(\cV)$ if
\begin{equation}
\int_0^\pi |f_1(\cos t)|^2 v_{11}(\cos t) + |f_2(\cos t)|^2 v_{22}(\cos t)  \, dt \, < \, \infty
\end{equation}
with 
\begin{equation}
\begin{split}
&v_{11}(\cos t) = v_{11}(\cos t; q^\tau, q^\si\mid q^2) = 
\frac{(q^2, -q^{2-2\tau};q^2)_\infty (e^{\pm 2 it};q^2)_\infty}{2\pi (-q^{2\tau};q^2)_\infty
 (q^{1+\si-\tau} e^{\pm it}, -q^{1-\si-\tau} e^{\pm it};q^2)_\infty} \\
&v_{22}(\cos t) = v_{11}(\cos t; q^{-\tau}, q^{-\si}\mid q^2)
\end{split}
\end{equation}
and so $v_{12}(\cos t)=0=v_{21}(\cos t)$.

In order to write down the orthogonality measure for the $2\times 2$-matrix-valued 
orthogonal polynomials from Theorem \ref{thm:genorth2t2MVOP} we need to calculate $Uf_k$ 
for $k=0$ and $k=1$. Expanding the standard orthonormal basis into the basis 
$\{ w^\phi_m, u^\phi_m\}_{m=0}^\infty$ as in \cite[p.~410]{KoelV}, and applying $U$, 
which is given by 
$(\La_1,\La_2)$ as in \cite[p.~411]{KoelV}, we get after a straightforward calculation that 
$Uf_k(\la) = \bigl( U_1f_k(\la), U_2f_k(\la)\bigr)^t$ with 
\begin{equation}\label{eq:rhostexplicitUonfk}
\begin{split}
U_1f_k(\la) \, &=\, \sum_{m=0}^\infty \frac{i^{-k}e^{-ik\phi} p_k(-q^{2m})}{\| v^\phi_{-q^{2m}}\|} 
e^{2im\phi}\, h_m(\la; q^\tau, q^\si\mid q^2) \\
U_2f_k(\la) \, &=\, \sum_{m=0}^\infty 
\frac{i^{-k}e^{-ik\phi} p_k(q^{2\tau+2m})}{\| v^\phi_{q^{2m+2\tau}}\|} 
e^{2im\phi}\, h_m(\la; q^{-\tau}, q^{-\si}\mid q^2) 
\end{split}
\end{equation}
where we have used the notation as in \cite[Prop.~5.2, p.~410]{KoelV} for the 
length of the vector, the Al-Salam--Carlitz polynomials 
$p_k(\cdot)$ and the Al-Salam--Chihara polynomials $h_m(\cdot)$. 

For $k=0$ we can use the generating function, see e.g. \cite[(3.8.14)]{KoekS}, directly
to find
\begin{equation}\label{eq:rhostUf0explicit}
\begin{split}
&F_{1,0}(\cos t; q^\tau, q^\si\mid q^2) = (U_1f_0)(\cos t) \\ =&\,   
\frac{1}{(-q^{2\tau};q^2)_\infty^\hf (qe^{i(t+2\phi)};q^2)_\infty}
\, \rphis{2}{1}{q^{1+\si-\tau}e^{it},-q^{1-\si-\tau}e^{it}}{-q^{2-2\tau}}{q^2,qe^{i(2\phi-t)}}
\end{split}
\end{equation}
and $(U_2f_0)(\cos t)=F_{1,0}(\cos t; q^{-\tau}, q^{-\si}\mid q^2)$ 
is obtained from $(U_1f_0)(\cos t)$ by replacing $(\si,\tau)$ by $(-\si,-\tau)$. 

For $k=1$ we have to take a linear combination. First, note, in the notation of \cite[Prop.~5.2]{KoelV},
$p_1(x) = q^{-\tau}(1-q^2)^{-\hf}(x+1-q^{2\tau})$, so that $p_1(-q^{2m})=-p_1(q^{2m+2\tau})$. 
In particular, we obtain, also using \cite[Prop.~5.2]{KoelV}, that $(U_2f_1)(\cos t)$ 
is obtained from $(U_1f_1)(\cos t)$ by switching $(\si,\tau)$ to $(-\si,-\tau)$ and multiplying 
by $-1$. 
Using the same generating function for the Al-Salam--Chihara polynomials twice we obtain
\begin{equation}\label{eq:rhostUf1explicit}
\begin{split}
&F_{1,1}(\cos t; q^\tau, q^\si \mid q^2) = (U_1f_1)(\cos t) \, = \\  
&\frac{-ie^{-i\phi}q^{-\tau}}{\sqrt{(1-q^2)(-q^{2\tau};q^2)_\infty}}
\Bigl( \frac{-1}{(q^3 e^{i(t+2\phi)};q^2)_\infty} 
\rphis{2}{1}{q^{1+\si-\tau}e^{it},-q^{1-\si-\tau}e^{it}}{-q^{2-2\tau}}{q^2,q^3e^{i(2\phi-t)}} \\
& \qquad \qquad + \frac{(1-q^{2\tau})}{(q e^{i(t+2\phi)};q^2)_\infty} 
\rphis{2}{1}{q^{1+\si-\tau}e^{it},-q^{1-\si-\tau}e^{it}}{-q^{2-2\tau}}{q^2,qe^{i(2\phi-t)}} \Bigr)
\end{split}
\end{equation}
and $(U_2f_1)(\cos t) = -F_{1,1}(\cos t;q^{-\tau}, q^{-\si}\mid q^2)$.

\begin{prop}\label{prop:rhostMVOP}
Consider the matrix-valued polynomials $P_n$ generated by 
\eqref{eq:3trr2by2MVOP} with initial conditions $P_{-1}(\la)=0$, $P_0(\la)=I$ and 
where the entries of the matrices $A_n$ and $B_n$ are given by \eqref{eq:rhostanbncc} 
with $\si\geq \tau$, $\si\geq -\tau$. Assume moreover $\si+\tau\leq 1$, $\si-\tau\leq 1$, 
then the matrix-valued 
 polynomials $P_n$ satisfy the orthogonality relations
\[
\begin{split}
& \int_0^\pi P_n(\cos t) W_2(\cos t) P_m(\cos t)^\ast\, dt \, = \, \de_{nm}I, \\
& W_2(\cos t)_{11}\, =\, |F_{1,0}(\cos t;q^{\tau}, q^{\si}\mid q^2)|^2 
v_{11}(\cos t; q^\tau, q^\si) + \bigl( (\si,\tau) \leftrightarrow (-\si,-\tau)\bigr)\\
& W_2(\cos t)_{21}\, =\, W_2(\cos t)_{12}\, =\, F_{1,0}(\cos t;q^{\tau}, q^{\si}\mid q^2)
 v_{11}(\cos t; q^\tau, q^\si) 
F_{1,1}(\cos t;q^{\tau}, q^{\si}\mid q^2) \\ & \qquad\qquad -  \bigl( (\si,\tau) 
\leftrightarrow (-\si,-\tau)\bigr)\\
& W_2(\cos t)_{22}\, =\, |F_{1,1}(\cos t;q^{\tau}, q^{\si}\mid q^2)|^2 
v_{11}(\cos t; q^\tau, q^\si) + \bigl( (\si,\tau) \leftrightarrow (-\si,-\tau)\bigr)
\end{split}
\]
where the functions on the right hand side are defined by 
\eqref{eq:rhostUf0explicit}, \eqref{eq:rhostUf1explicit} and the 
notation $\bigl( (\si,\tau) \leftrightarrow (-\si,-\tau)\bigr)$ means 
that we have to add the same term but with parameters 
$(\si,\tau)$ replaced by $(-\si,-\tau)$. 
\end{prop}

Note that $W_2(\cos t)_{ij}$ is explicit as a sum of $i+j$ terms, 
each term being a product of two ${}_2\vp_1$-series.

In case the assumption $\si+\tau\leq 1$, $\si-\tau\leq 1$ is dropped we 
obtain a finite discrete set of mass points in the orthogonality relations of 
Proposition \ref{prop:rhostMVOP}, and the weight $W_1$ at these points can be 
calculated in the same way  from Theorem \ref{thm:genorth2t2MVOP}. Alternatively, 
they can be obtained from writing the integral of Proposition \ref{prop:rhostMVOP} 
as a contour integral, and then shifting contours which leads to discrete masses at 
the poles with weights given in terms of residues analogous to the case of the 
Askey-Wilson polynomials, see \cite{AskeW}.

\emph{Acknowledgement.} 
The research of Mourad E.H. Ismail is supported by a Research Grants
Council of
Hong Kong  under contract \# 101411 and the  NPST Program of King Saud
University, project number 10-MAT1293-02,

This work was partially supported by a grant from the `Collaboration Hong Kong - Joint 
Research Scheme' sponsored by the Netherlands Organisation of Scientific Research and 
the Research Grants Council fo Hong Kong (Reference number: 600.649.000.10N007).


\end{document}